\documentclass[a4paper,12pt]{amsart}
	
\usepackage{fullpage}
\usepackage{amsfonts}
\usepackage{amssymb,amsthm,amsmath,color}
\usepackage{array}
\usepackage{mathrsfs}
\usepackage{dsfont}
\usepackage{lmodern}
\usepackage{bbm}
\usepackage{hyperref}
\usepackage{graphicx}
\usepackage[utf8]{inputenc}  
\usepackage[T1]{fontenc} 
\usepackage{stmaryrd}  
\usepackage{tikz}
\usepackage{tkz-euclide}
\usepackage[backend=biber]{biblatex}
\numberwithin{equation}{section}

\DeclareMathAlphabet{\mathbbo}{U}{bbold}{m}{n}
\newcommand{\1}{\mathbbo{1}}

\newtheorem{theorem}{Theorem}[section]

\newtheorem{lem}[theorem]{Lemma}

\newenvironment{rmq}{\stepcounter{theorem}\noindent\textbf{Remark \thetheorem}.}{}

\newcommand{\N}{\mathbb N}
\newcommand{\Z}{\mathbb Z}
\newcommand{\R}{\mathbb R}

\newcommand{\dd}{\text{d}}
\renewcommand{\mod}{\,\mathrm{mod}\,}

\begin{document}

\title{On the largest Sidon subset in a finite subset of $\R^N$}

\author{A. Bailleul}
\address{ENS Paris-Saclay, Centre Borelli, UMR 9010, 91190 Gif-sur-Yvette, France}
\email{alexandre.bailleul@ens-paris-saclay.fr}

\author{R. Riblet}
\address{ENS Paris-Saclay, Centre Borelli, UMR 9010, 91190 Gif-sur-Yvette, France}
\email{robin.riblet@ens-paris-saclay.fr}

    \maketitle

\begin{abstract}
We obtain a new lower bound on the largest Sidon subset of an arbitrary finite set of integers. If $H(n)$ denotes the minimum, over all $n$-element subsets of $\mathbb Z$, of the largest Sidon subset they contain, we prove that
$H(n)\geqslant \left(\frac{1}{3\sqrt 3}+o(1)\right)\sqrt n \gtrsim 0.19\sqrt n$. This improves a lower bound of Abbott related to a conjecture of Erd\H{o}s on Sidon subsets of arbitrary sets of integers. 
The main ingredient is a compression lemma which produces, from any finite set of integers, a large subset admitting an injective Freiman $2$-morphism into a cyclic group. Combined with Singer's covering of $\mathbb Z/(q^2+q+1)\mathbb Z$ by Sidon sets, this yields the stated bound.

We further extend the result to finite subsets of $\mathbb R^N$, uniformly in the dimension, by means of a projection argument and a Dirichlet approximation preserving Sidon's equation. As a consequence, every set of $n$ points in $\mathbb R^N$ contains a Sidon subset of cardinality at least $\left(\frac{1}{3\sqrt 3}+o(1)\right)\sqrt n$.
We also discuss an adaptation to $B_2[g]$ sets, obtaining a lower bound of order $\frac{1}{3\sqrt 3}\sqrt{gn}$, and explain how the method can be adapted to other linear additive constraints.
\end{abstract}
		
\section{Introduction}

A Sidon set, also called a $B_2$-set, is a subset of an abelian semigroup with the property that all sums of two elements are distinct; equivalently, it is a set with no non-trivial solution to the equation $x+y=z+t$. Sidon sets were first studied by Simon Sidon \cite{Simon_Sidon} in connection with Fourier series. He investigated the size of the largest Sidon subset of $\left\llbracket1,n\right\rrbracket$. This question has since been extensively studied, and it is now well known (see \cite{HalberstamRoth} or \cite{OBryant}) that the maximum size of a Sidon set in an interval of length $n$ is asymptotic to $\sqrt{n}$. We denote this maximum by $s\left(\left\llbracket 1,n\right\rrbracket\right)$. The lower bound was obtained independently by Chowla \cite{Chowla} and Erd\H{o}s \cite{ErdosTuran}, who established
\begin{equation}\label{Minsn}
   \liminf\limits_{\substack{n\rightarrow +\infty}}\dfrac{s\left(\left\llbracket 1,n\right\rrbracket\right)}{\sqrt{n}}\geqslant 1. 
\end{equation}
For the upper bound, Erd\H{o}s and Tur\'an \cite{ErdosTuran} proved that $s\left(\left\llbracket 1,n\right\rrbracket\right)<\sqrt{n}+O\left( n^{1/4}\right)$. This was later sharpened by Lindström \cite{Lindstrom_ameliore_E-T}, who proved that $s\left(\left\llbracket 1,n\right\rrbracket\right)<\sqrt{n}+n^{1/4}+1$. More recently, Carter, Hunter and O'Bryant \cite{OBryantCarterHunter} obtained
\begin{equation}\label{Majsn}
    s\left(\left\llbracket 1,n\right\rrbracket\right)<\sqrt{n}+0.98183n^{1/4}+O(1) .
\end{equation}
Thus, equations \eqref{Minsn} and \eqref{Majsn} imply that $s(n)\underset{n \to \infty}{\sim} \sqrt{n}$. It is natural to ask what happens for ambient sets other than intervals. For arithmetic progressions, the same result follows immediately from the homogeneity of Sidon's equation. For the union of two intervals $I_1$ and $I_2$, \cite{Robin} proves that
$$(0.876+o(1))\sqrt{n}\leqslant s(I_1\sqcup I_2) \leqslant (1+o(1))\sqrt{n},$$
where $n=|I_1\sqcup I_2|$. Beyond these cases, little seems to be known. For instance, what is the maximum size of a Sidon set contained in $n/2$ dominoes? In the first $n$ primes? Or in the first $n$ squares? For this last example, stronger bounds than the one obtained in the present article are known; see \cite[Section 6]{LefmannThiele}.

Erd\H{o}s conjectured that intervals of integers are the least favorable ambient sets to produce Sidon sets, since they have the largest additive energy. More precisely, if we define
\begin{equation}H(n)=\min\limits_{\substack{E\subset\Z \\ |E|=n}}\max\left\lbrace |S| : S\subset E, S \text{ is a Sidon set}\right\rbrace,\end{equation}
then he conjectured that $H(n) \underset{n \to \infty}{\sim} \sqrt{n}$. The upper bound follows immediately from Erd\H{o}s and Tur\'{a}n’s result. For a long time, the best lower bound for $H(n)$ was due to Mian and Chowla \cite{MianChowla}, who proved that $H(n)>cn^{1/3}$ for some positive constant $c$. This was improved 30 years later by Koml\'{o}s, Sulyok and Szemerédi \cite{Sulyok}, who partially proved the conjecture by obtaining the correct order of magnitude, namely $H(n)>(0.0000305+o(1))\sqrt n$. The proof proceeds through successive reductions via Euclidean division, carefully passing to suitable subsets of residues where Sidon's equation is preserved. The need to discard some residues at each step is what causes the losses in the constant. Fifteen years later, using the same general proof structure but incorporating a construction due to Singer, Abbott \cite{Abbott} further improved this bound to $H(n)>(0.0805+o(1))\sqrt n$. This had remained the best known result until now. The main result of this paper is the following:\\

\noindent\textbf{Theorem \ref{AmelioAbbott}.}
\textit{We have $$H(n)\geqslant\left(\frac{1}{3\sqrt 3}+o(1)\right)\sqrt n \gtrsim 0.19\sqrt n.$$}

This improvement is achieved by replacing the previous reduction steps with a single step, making use of the existence of a Freiman 2-morphism into a cyclic group.\\

We also extend this result to finite subsets of $\R^N$. For every $N\geqslant 1$, define
\begin{equation}F_N(n)=\min\limits_{\substack{E\subset\R^N \\ |E|=n}}\max\left\lbrace |S| : S\subset E, \ S \text{is a Sidon set}\right\rbrace.\end{equation} In this more general setting, we obtain the following:\\

\noindent\textbf{Theorem \ref{MainResult}.}
\textit{We have $$F_N(n)\geqslant\left(\frac{1}{3\sqrt 3}+o(1)\right)\sqrt n \gtrsim 0.19\sqrt n.$$}

It is interesting to note that this bound does not depend on the dimension $N$. This follows from a simple projection trick (see Lemma \ref{proj}) that implies that $F_N(n) = F_1(n)$. The connection between Theorem \ref{MainResult} and Theorem \ref{AmelioAbbott} is established by a Dirichlet approximation that is sufficiently fine to ensure that Sidon’s property is preserved. In particular, in Lemma \ref{Dirichlet} we prove that $F_1(n) = H(n)$.\\

Theorem \ref{MainResult} has several immediate applications. In particular, it implies that from any set of $n$ points in the plane one can always extract $\lfloor 0.19\sqrt{n}\rfloor$ points containing no parallelogram when $n$ is large enough. Indeed, $ABCD$ is a parallelogram if and only if $\overrightarrow{AB}=\overrightarrow{DC}$, equivalently $B-A=C-D$, which is forbidden in a Sidon set.
Such parallelogram-free point sets (a much-studied subject, see \cite{BrassMoserPach, Eppstein,ParallelogramFree} for example), also described as point sets determining only distinct vectors, appear naturally in questions of Erd\H{o}s, Hickerson and Pach \cite{ErHiPa} on distinct distances. 

For $N=1$, being a Sidon set is also equivalent to having all pairwise distances distinct. Thus Theorem \ref{MainResult} also shows that from any set of $n$ real numbers, one can always extract $\lfloor 0.19\sqrt{n}\rfloor$ real numbers determining only distinct distances when $n$ is large enough. This improves the previous best result, due to Abbott \cite{Abbott}. This problem was first posed by Erd\H{o}s and has been extensively studied (for example, see \cite{ErdosDistance1,ErdosDistance2,ErdosDistance3, dumiDistance, GuthKatz}). Very recently, Tao \cite{TaoDistance} solved one of Erd\H{o}s's favourite problems on distinct distances by combining distance questions with the exclusion of certain four-point configurations. This illustrates that Sidon-type conditions, and more generally constraints forbidding additive or geometric configurations of small size, remain closely connected to current developments in discrete geometry (see also \cite{dumiDistance2}).

In any dimension, the failure of the Sidon property is equivalent to the existence of two distinct pairs of points with the same midpoint. Thus Theorem \ref{MainResult} implies that from any set of $n$ points in $\mathbb{R}^N$, one can extract $\lfloor 0.19\sqrt n \rfloor$ points whose pairwise midpoints are all distinct when $n$ is large enough.\\

In the final section of this article, we extend Theorem \ref{MainResult} to $B_2[g]$ sets, which are defined by the fact that their pairwise sums can be represented by at most $g$ unordered pairs. By definition, Sidon sets are $B_2[1]$ sets. Defining \begin{equation}F_{g,N}(n)=\min\limits_{\substack{E\subset\R^N \\ |E|=n}}\max\left\lbrace |S| : S\subset E, S \text{ is a } B_2[g]\text{ set}\right\rbrace,\end{equation} we prove the following result:\\ 

\noindent \textbf{Theorem \ref{généB2}.}
For every fixed $N,g\geqslant 1$, we have
    $$F_{g,N}(n)\geqslant\left(\frac{\sqrt{g}}{3\sqrt 3}+o(1)\right)\sqrt n \gtrsim 0.19\sqrt{gn}.$$
    
To this end, we prove in particular a generalization of Singer's covering by Sidon sets to the $B_2[g]$ setting; see Proposition \ref{SingerB2}. We conclude the paper by explaining how the method can be adapted to other linear additive constraints beyond Sidon's equation.

\section{Large Sidon subsets}

\begin{theorem}\label{MainResult}
For every $N\geqslant 1$, we have
    $$F_N(n)\geqslant\left(\frac{1}{3\sqrt 3}+o(1)\right)\sqrt n \gtrsim 0.19\sqrt n,$$
    where
    $$F_N(n)=\min\limits_{\substack{E\subset\R^N \\ |E|=n}}\max\left\lbrace |S| : S\subset E, S \text{ is a Sidon set}\right\rbrace.$$
\end{theorem}

The proof of this result is divided into two steps. The first and most important one is to establish the result for finite subsets of $\Z$; then, for the second step, we reduce to this setting by using a projection trick from $\R^N$ to $\R$ and then an approximation by rationals, so that by clearing denominators we are back in the integers.

\subsection{Step 1: Result in $\Z$}

We begin by proving the result for finite subsets of $\Z$.

\begin{theorem}\label{AmelioAbbott}
    We have
     $$H(n)\geqslant\left(\frac{1}{3\sqrt 3}+o(1)\right)\sqrt n \gtrsim 0.19\sqrt n,$$
    where
    $$H(n)=\min\limits_{\substack{E\subset\Z \\ |E|=n}}\max\left\lbrace |S| : S\subset E, S \text{ is a Sidon set}\right\rbrace.$$
\end{theorem}

The proof of this Theorem relies on Lemma \ref{Freiman} below. Recall that a Freiman $2$-morphism is a map $\phi$ between two subsets $X$ and $Y$ of abelian groups such that, for all $x,y,z,t\in X$,
$$x+y=z+t \Rightarrow \phi(x)+\phi(y)= \phi(z)+\phi(t).$$
This notion allows the Sidon data to be transported from $X$ to $Y$. Lemma \ref{Freiman} replaces the reduction Lemmas A to D from \cite{Abbott}.

\begin{lem}\label{Freiman}
    Let $A\subset\Z$ be finite and let $m\in\N^*$. There exists $C\subseteq A$ with $|C|\geqslant \frac{|A|}{2}-\frac{|A|^2}{2m}$ and an injective Freiman $2$-morphism from $C$ to $\Z/m\Z$.
\end{lem}

For the proof of this lemma, we shall need a small, well-known intermediate lemma.

\begin{lem}\label{preservMesure}
    Let $E\subseteq[0,1[$ be a Borel set, let $k\in\Z^*$, and let $T_k:x\mapsto kx\mod 1$. Then $\lambda(E)=\lambda(T_k^{-1}(E)),$ where $\lambda$ denotes Lebesgue measure.
\end{lem}
\begin{proof}
    Let $I$ be an interval of $[0,1[$. If $k=1$, there is nothing to prove. Let $k\geqslant 2$. We have
    $$T_k^{-1}(I)=\bigcup_{j=0}^{k-1}\frac{I+j}{k} \mod 1,$$
    in other words, $T_k^{-1}(I)$ is the union of $k$ disjoint intervals, each of length $\lambda(I)/k$. 
   
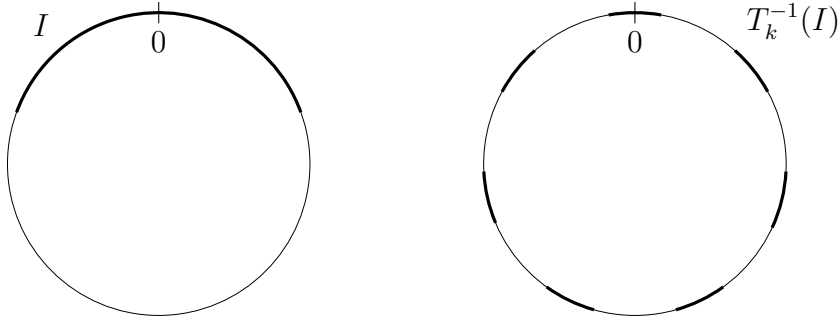
    \begin{center}
        \begin{tikzpicture}[scale=2]
        \tkzDefPoint(0,0){O}\tkzDefPoint(0,1){A}\tkzDefPoint(1,0){B}\tkzDefPoint(40:1){C}
        \tkzDefPoint(160:1){I1}
        \tkzDefPoint(130:1){U}
        \tkzDefPoint(20:1){I2}
            \tkzDrawCircle[black](O,A)
            \tkzDrawArc[very thick](O,I2)(I1)
            \tkzLabelPoint[below=0.1](A){$0$}
             \tkzLabelPoint[above left](U){$I$}
            \tkzMarkArc[mark=|](O,I2,I1)
        \end{tikzpicture}
   \hspace{2cm}
        \begin{tikzpicture}[scale=2]
        \tkzDefPoint(0,0){O}\tkzDefPoint(0,1){A}\tkzDefPoint(1,0){B}\tkzDefPoint(40:1){C}
         \tkzDrawCircle[black](O,A)
        \tkzDefPoint(100:1){I1}
        \tkzDefPoint(80:1){I2}
            \tkzDrawArc[very thick](O,I2)(I1)
            \tkzMarkArc[mark=|](O,I2,I1)
             \tkzLabelPoint[below=0.1](A){$0$}
              \tkzDefPoint(151.43:1){I3}
        \tkzDefPoint(131.43:1){I4}
            \tkzDrawArc[very thick](O,I4)(I3)
                \tkzDefPoint(202.86:1){I5}
        \tkzDefPoint(182.86:1){I6}
            \tkzDrawArc[very thick](O,I6)(I5)
                \tkzDefPoint(254.29:1){I7}
        \tkzDefPoint(234.29:1){I8}
            \tkzDrawArc[very thick](O,I8)(I7)
                \tkzDefPoint(305.72:1){I9}
        \tkzDefPoint(285.72:1){I10}
            \tkzDrawArc[very thick](O,I10)(I9)
                \tkzDefPoint(357.15:1){I11}
        \tkzDefPoint(335.15:1){I12}
            \tkzDrawArc[very thick](O,I12)(I11)
                \tkzDefPoint(408.58:1){I13}
        \tkzDefPoint(388.58:1){I14}
            \tkzDrawArc[very thick](O,I14)(I13)
           \tkzLabelPoint[above right](I13){$T_k^{-1}(I)$}
        \end{tikzpicture}
    \end{center}

    Thus $\lambda(T_k^{-1}(I))=k\frac{\lambda(I)}{k}=\lambda(I)$. Since intervals generate the Borel sets, we obtain $\lambda(T_k^{-1}(E))=\lambda(E)$ for every Borel set $E$ and every $k\geqslant 1$.
    The case $k\leqslant -1$ follows by composing with the symmetry $k\mapsto -k$, which also preserves Lebesgue measure.
\end{proof}



We can now begin the proof of Lemma \ref{Freiman}.

\begin{proof}[Proof of Lemma \ref{Freiman}]
    Let $A\subset\Z$ be finite and write $|A|=n$. Let also $m\in\Z_{> 0}$ and $\theta\in [0,1[$. For $a \in A$, set
    $$\phi_\theta (a)=\left\lfloor am\theta\right\rfloor \mod m,$$
    and
    $$B_\theta =\left\lbrace a\in A : \left\lbrace am\theta\right\rbrace<\frac12   \right\rbrace,$$
    where $\left\lfloor \cdot \right\rfloor$ denotes the integer part and $\left\lbrace \cdot \right\rbrace$ the fractional part of a real number. For all $b,b'\in B_\theta$, on the one hand we have
    $$\phi_\theta(b)+\phi_\theta(b')=\left\lfloor bm\theta\right\rfloor+\left\lfloor b'm\theta\right\rfloor \mod m.$$
    On the other hand, since $x=\left\lfloor x\right\rfloor+\left\lbrace x\right\rbrace$ for every real number $x$, we have
      \begin{align*}\phi_\theta(b+b')&=\left\lfloor (b+b')m\theta\right\rfloor \mod m\\ &=\left\lfloor \left\lfloor bm\theta\right\rfloor+\left\lfloor b'm\theta\right\rfloor+\left\lbrace bm\theta\right\rbrace+\left\lbrace b'm\theta\right\rbrace\right\rfloor \mod m\\ &= \left\lfloor bm\theta\right\rfloor+\left\lfloor b'm\theta\right\rfloor+\left\lfloor\left\lbrace bm\theta\right\rbrace+\left\lbrace b'm\theta\right\rbrace\right\rfloor \mod m\\ &= \left\lfloor bm\theta\right\rfloor+\left\lfloor b'm\theta\right\rfloor\mod m,\end{align*}

because $0 \le \left\lbrace bm\theta\right\rbrace+\left\lbrace b'm\theta\right\rbrace< 1 $, and therefore $\left\lfloor\left\lbrace bm\theta\right\rbrace+\left\lbrace b'm\theta\right\rbrace\right\rfloor=0$. Thus
 $$\phi_\theta(b)+\phi_\theta(b')= \phi_\theta(b+b'),$$ and so $\phi_\theta$ is indeed a Freiman $2$-morphism on $B_\theta$. We shall now reduce $B_\theta$ in order to ensure injectivity.
 
 We first bound from below the "expectation" of $|B_\theta|$. For every $a\in A$, set
 $$E_a=\left\lbrace \theta\in [0,1[ \ : \left\lbrace am\theta \right\rbrace <1/2\right\rbrace.$$
 Thus $|B_\theta|=\sum_{a\in A} \1_{E_a}(\theta)$ and
 $$\int_0^1 |B_\theta| \,\dd \theta=\sum_{a\in A} \int_0^1 \1_{E_a}(\theta) \,\dd \theta=\sum_{a\in A} \lambda (E_a).$$
Now, if $a\neq 0$, the map $\theta\mapsto ma\theta\mod 1$ preserves Lebesgue measure on the torus $\R/\Z$ by Lemma \ref{preservMesure}. Hence $\lambda(E_a)=1/2$. Moreover $E_0=[0,1[$, so $\lambda(E_0)=1$ and in particular we obtain
\begin{equation}\label{EspBteta}
    \int_0^1 |B_\theta| \,\dd \theta\geqslant \frac{n}{2}.
\end{equation}

 We now bound from above the expectation of the number of collisions through $\phi_{\theta}$. For all distinct $b,b'\in A$, set
 $$F_{b,b'}=\left\lbrace \theta\in [0,1[ \ : b,b'\in B_\theta \ , \ \phi_\theta(b)=\phi_\theta(b') \right\rbrace.$$
 We have $\lambda(F_{b,b'})\leqslant 1/m$. Indeed, if $b,b'\in B_\theta$, then we can write
 $$bm\theta = u+\alpha \quad \text{and} \quad b'm\theta=v+\beta,$$
 where $u,v\in\Z$ and $\alpha,\beta\in[0,1/2[$. Since $\phi_\theta(b)=\phi_\theta(b')$, we have $u \equiv v \mod m$, and hence $u-v\in m\Z$. Subtracting, we obtain $(b-b')m\theta=(u-v)+(\alpha-\beta)$, and therefore $(b-b')\theta\in \Z+\frac{\alpha-\beta}{m}$. Since $|\alpha-\beta|<1/2$, we get $\text{dist}((b-b')\theta,\Z)\leqslant \frac{1}{2m}$. Thus $$F_{b,b'}\subseteq \left\lbrace \theta\in [0,1[ \ : \text{dist}((b-b')\theta, \Z)\leqslant \frac{1}{2m} \right\rbrace,$$
 and since $b-b'\neq 0$, the map $\theta\mapsto (b-b')\theta \mod 1$ preserves Lebesgue measure by Lemma \ref{preservMesure}, hence
 $$\lambda(F_{b,b'})\leqslant\lambda\left( \left\lbrace x\in [0,1[ \ : \text{dist}( x, \Z)\leqslant \frac{1}{2m} \right\rbrace \right)=\frac{1}{m}.$$
 We now define $P_\theta$ as the number of unordered pairs the same fiber of $\phi_{\theta}$:
 \begin{align*}
     P_\theta & =|\left\lbrace \left\lbrace b,b'\right\rbrace\subset B_\theta :  b\neq b', \phi_\theta(b)=\phi_\theta(b')\right\rbrace | \\
     & = \sum\limits_{\substack{\left\lbrace b,b'\right\rbrace\subset A \\  b\neq b'}} \1_{F_{b,b'}}(\theta).
 \end{align*}
 We therefore have
 $$\int_0^1 P_\theta\,\dd \theta=\int_0^1 \sum\limits_{\substack{\left\lbrace b,b'\right\rbrace\subset A \\  b\neq b'}} \1_{F_{b,b'}}(\theta)\,\dd \theta = \sum\limits_{\substack{\left\lbrace b,b'\right\rbrace\subset A \\  b\neq b'}} \lambda(F_{b,b'})\leqslant \binom{n}{2}\frac{1}{m}.$$
 
 This gives the following upper bound for the expectation of the number of collisions:
 \begin{equation}\label{espColli}
     \int_0^1 P_\theta\,\dd \theta\leqslant\frac{n^2}{2m}.
 \end{equation}
This bound will allow us to control the number of elements that must be removed from $B_\theta$ in order to make $\phi_\theta$ injective. Indeed, let $f_1,\dots,f_t$ be the list of the $t$ non-empty fibres ($t\leqslant m$) of $\phi_{\theta}$ restricted to $B_{\theta}$, and let $r_1,\dots,r_t$ be their respective sizes. 
By choosing one element in each of these $t$ fibres, we obtain a subset $C_\theta\subseteq B_\theta$ of cardinality $t$ on which $\phi_\theta$ is injective. Moreover, since $r_1+\dots+r_t=|B_\theta|$, we have
$$|B_\theta\setminus C_\theta|=|B_\theta|-|C_\theta|=\sum_{i=1}^t (r_i-1).$$
On the other hand, $P_\theta=\displaystyle \sum_{i=1}^t\binom{r_i}{2}$ and $\binom{r_i}{2}\geqslant (r_i-1)$ for every $r_i\geqslant 1$, so
$|B_\theta|-|C_\theta|\leqslant P_\theta$, from which we get $$|C_\theta|\geqslant |B_\theta|-P_\theta.$$ Using \eqref{EspBteta} and \eqref{espColli}, we obtain
$$\int_0^1|C_\theta| \,\dd \theta\geqslant \int_0^1(|B_\theta|-P_\theta)\,\dd \theta\geqslant \frac{n}{2}-\frac{n^2}{2m}.$$
Finally, there exists $\theta\in [0,1[$ such that $|C_\theta|\geqslant \frac{n}{2}-\frac{n^2}{2m}$ and such that $\phi_\theta$ is an injective Freiman $2$-morphism from $C_\theta$ to $\Z/m\Z$, which concludes the proof.
\end{proof}

We shall combine Lemma \ref{Freiman} with the Singer covering trick (see Lemma \ref{LemmeSinger} below and \cite{Singer} for its very elegant proof based on the existence of a collineation in the projective plane) in order to prove Theorem \ref{AmelioAbbott}.

\begin{lem}[Singer]\label{LemmeSinger}
    Let $q$ be a prime power. There exist $q+1$ Sidon sets, each of cardinality $q+1$, which cover $\Z/(q^2+q+1)\Z$.
\end{lem}

We can now prove Theorem \ref{AmelioAbbott}.

\begin{proof}[Proof of Theorem \ref{AmelioAbbott}]
Let $A\subset \Z$ be finite, with cardinality $n\in\N^*$. Lemma \ref{Freiman} will serve as a "compression" step before applying Singer's covering to find, by averaging, a large Sidon set containing many elements of $A$. For the compression to be efficient, we would like to "lose" as few elements as possible in terms of $\frac{|A|^2}{2m}$, and therefore to choose $m$ as large as possible. Conversely, if $m$ is too large, then $\Z/m\Z$ is too large, and the Singer covering will be too spread out, reducing the concentration of the compressed elements inside one of the Sidon sets in the covering. The right choice is to take $m$ of order $|A|$, that is, $m=(c+o(1))n$, with $c>1$ in order to keep a positive proportion of the elements. Let therefore $c>1$.

    We first recall that, as $x\rightarrow +\infty$, there exists a prime number $p=p(x)$ such that $p^2+p+1=(1+o(1))x$. Indeed, $\sqrt{x}$ lies between two consecutive prime numbers $p_1$ and $p_2$, hence
    $$\frac{p_1}{p_2}\leqslant \frac{\sqrt{x}}{p_2}\leqslant 1,$$
    and by the Prime Number Theorem, $\frac{p_1}{p_2}$ tends to $1$, so $p_2=\sqrt{x}(1+o(1))$. It follows that $p_2^2=x(1+o(1))$ and hence $p_2^2+p_2+1=x(1+o(1))$.

    Take $x=cn$. From the previous argument, there exists a prime number $p$ such that $p^2+p+1=cn(1+o(1))$. Set $m=p^2+p+1$. By Lemma \ref{Freiman}, there exists $C\subseteq A$ with $|C|=\left( \frac{1}{2}-\frac{1}{2c}+o(1)\right) n$ and an injective Freiman $2$-morphism $\phi$ from $C$ into $\Z/m\Z$.

    Since $m=p^2+p+1$, by Lemma \ref{LemmeSinger}, we can cover $\Z/m\Z$ by $p+1$ Sidon sets $\left\lbrace S_i\right\rbrace_{i=1}^{p+1}$, each of cardinality $p+1$. By the pigeonhole principle, there exists $S\in \left\lbrace S_i\right\rbrace_{i=1}^{p+1}$ such that $|S\cap \phi(C)|\geqslant \frac{|\phi(C)|}{p+1}$. By injectivity, $|\phi(C)|=|C|=\left( \frac{1}{2}-\frac{1}{2c}+o(1)\right) n$, and $p+1=\sqrt{cn}(1+o(1))$, hence
    $$|S\cap \phi(C)|\geqslant \left(\frac{c-1}{2c\sqrt{c}}+o(1)\right) \sqrt{n}.$$
    Set $S_C=\phi^{-1}(S\cap\phi(C))$. This is a subset of $C$, and hence of $A$, with cardinality greater than $\left(\frac{c-1}{2c\sqrt{c}}+o(1)\right) \sqrt{n}$. Moreover, if there exist $a,b,c,d\in S_C$ such that $a+b=c+d$, then $\phi(a)+\phi(b)=\phi(c)+\phi(d)$ since $\phi$ is a Freiman $2$-morphism. But $\phi(a),\phi(b),\phi(c),\phi(d)\in S$ by the definition of $S_C$, and $S$ is a Sidon set, so necessarily $\left\lbrace \phi(a),\phi(b)\right\rbrace=\left\lbrace \phi(c),\phi(d)\right\rbrace$. Again by injectivity, $\left\lbrace a,b\right\rbrace=\left\lbrace c,d\right\rbrace$, and therefore $S_C$ is a Sidon set.

    Finally, the derivative of $f:c\mapsto \frac{c-1}{2c\sqrt{c}}$ is $f'(c)=(3-c)\frac{\sqrt{c}}{4c^3}$, so the maximum of $f$ is attained at $c=3$. Thus
    $$H(n)\geqslant |S_3|\geqslant (f(3)+o(1))\sqrt n=\left( \frac{1}{3\sqrt{3}}+o(1)\right)\sqrt n.$$
\end{proof}

\subsection{Step 2: General case}

To handle the general case in $\R^N$, the idea is to project the point cloud onto a real line, and then to use a sufficiently fine Dirichlet approximation in order to reduce to rational numbers with a common denominator while preserving the Sidon data. The following lemma allows us to project onto the real line.

\begin{lem}\label{proj}
    Let $x_1,\dots,x_n$ be distinct points of $\R^N$. There exists a line $d\subset \R^N$ such that $\vert\left\lbrace p_d(x_1),\dots,p_d(x_n)\right\rbrace\vert=n$, where $p_d$ denotes the orthogonal projection onto $d$.
\end{lem}
\begin{proof}
    It is enough to observe that $p_d(x_i)=p_d(x_j)$ if and only if $d$ is orthogonal to $(x_i,x_j)$. Consider all the linear hyperplanes $\left\lbrace\left\langle\overrightarrow{x_ix_j}\right\rangle^\perp\right\rbrace_{1\leqslant i<j\leqslant n}$. There are at most $\binom{n}{2}<+\infty$ of them, and therefore
    $$\lambda_N\left(\bigcup_{1\leqslant i<j\leqslant n}\left\langle\overrightarrow{x_ix_j}\right\rangle^\perp \right)=0,$$
    where $\lambda_N$ denotes Lebesgue measure on $\R^N$ (viewed as a vector space).
    In particular, there exists $\overrightarrow{u}\notin \bigcup_{1\leqslant i<j\leqslant n}\left\langle\overrightarrow{x_ix_j}\right\rangle^\perp$, and then $d=\left\langle\overrightarrow{u}\right\rangle$ works.
\end{proof}

Going from $\R$ to $\Z$ while preserving the Sidon data through Dirichlet approximation is well known (see for instance \cite{hardyWright}), but for the sake of completeness, and because of a few technical subtleties, we give the details below.

\begin{lem}\label{Dirichlet}
We have $F_1(n) = H(n)$.
\end{lem}

\begin{proof}
Since $\Z \subset \R$, we clearly have $F_1(n) \leq H(n)$. To prove the other inequality, let $A=\{a_1, \dots, a_n\}\subset\R$ and let $M>4$ be an integer. The condition $M>4$ is here to ensure the transfer of the Sidon property. For every $0\leqslant k\leqslant M^n$, consider the points $\mathfrak{a}_k=(\left\lbrace ka_i\right\rbrace )_{1 \le i \le n}$. These are points of the hypercube $[0,1[^n$, which decomposes into $M^n$ small hypercubes of side length $1/M$. By the pigeonhole principle, one of them contains two points $\mathfrak{a}_k$ and $\mathfrak{a}_{k'}$, which means that $|\left\lbrace ka_i\right\rbrace-\left\lbrace k'a_i\right\rbrace|<1/M$ for every $i\in\left\lbrace 1,\dots,n\right\rbrace$. Since
    $|\left\lbrace ka_i\right\rbrace-\left\lbrace k'a_i\right\rbrace|=|ka_i-\lfloor ka_i\rfloor-k'a_i+\lfloor k'a_i\rfloor|$, we get $\text{dist}(m_Ma_i,\Z)<1/M$ for every $i\in\left\lbrace 1,\dots,n\right\rbrace$ (where we have set $m_M=|k-k'|$). Thus, for every $M>4$ and every $i\in\left\lbrace 1,\dots,n\right\rbrace$, there exist $m_M\in\N^*$ and $r_{M,i}\in\Z$ such that $|m_Ma_i-r_{M,i}|<1/M$.
    
    Since this is true for arbitrarily large $M$, we shall choose $M$ in such a way that the $\left\lbrace r_{M,i}\right\rbrace_{i=1}^n$ are all distinct.

    Without loss of generality, assume that $a_1<\dots<a_n$ and set $\Delta=\min \left\lbrace a_{i+1}-a_i\right\rbrace$. Let $M>2/\Delta$ and let $i\in\left\lbrace 1,\dots,n-1\right\rbrace$. We have
    \begin{align*}
    r_{M,i+1}-r_{M,i} & =a_{i+1}m_M-a_i{m_M}+\left(r_{M,i+1}-a_{i+1}{m_M}\right)+\left( a_i{m_M}-r_{M,i}\right) \\
    & \geqslant (a_{i+1}-a_i)m_M -\left|r_{M,i+1}-a_{i+1}{m_M}\right|-\left|a_i{m_M}-r_{M,i}\right| \\
    & \geqslant \Delta{m_M}-\frac{2}{M}\geqslant \Delta-\frac{2}{M}>0,
    \end{align*}
    and therefore $r_{M,1}<\dots<r_{M,n}$. In particular, they are all distinct.
    
    Choose now $M>\max(4,2/\Delta)$ and let $\left\lbrace r_{M,j}\right\rbrace_{j\in J}$ be a Sidon set contained in $\left\lbrace r_{M,i}\right\rbrace_{i=1}^n$ ($J\subseteq \left\lbrace 1,\dots, n\right\rbrace$). Then $\left\lbrace a_j\right\rbrace_{j\in J}$ is a Sidon set. Indeed, let $i,h,k,l\in J$ be such that $a_i+a_h=a_k+a_l$. We have
     \begin{align*}
        |r_{M,i}+r_{M,h}-r_{M,k}-r_{M,l}| & =|(r_{M,i}-m_Ma_i)+(r_{M,h}-m_Ma_h)\\&\quad +(m_Ma_k-r_{M,k})+(m_Ma_l-r_{M,l})| \\
        & < 4/M \\ & <1.
    \end{align*}
Now $r_{M,i}+r_{M,h}-r_{M,k}-r_{M,l}$ is an integer, so it must be zero. But since $\left\lbrace r_{M,j}\right\rbrace_{j\in J}$ is a Sidon set, we have
$\left\lbrace r_{M,i},r_{M,h}\right\rbrace=\left\lbrace r_{M,k},r_{M,l}\right\rbrace$, and therefore $\{i, k\} = \{h, l\}$ since the $r_{M,j}$'s are pairwise distinct. In particular $\left\lbrace a_{i},a_{h}\right\rbrace=\left\lbrace a_{k},a_{l}\right\rbrace$ and we obtained that $\left\lbrace a_{j}\right\rbrace_{j\in J}$ is also a Sidon set, which concludes the proof.
\end{proof}

\begin{proof}[Proof of Theorem \ref{MainResult}] Let $A = \{x_1,\dots,x_n\}$ be a set of $n$ distinct points of $\R^N$. By Lemma \ref{proj}, there exists a line $d\subset \R^N$ such that $\vert\left\lbrace p_d(A)\right\rbrace\vert=n$. In particular, a Sidon subset $S$ of $p_d(A)$ corresponds to a Sidon subset of $A$. Indeed, for $x,y,z,t\in A$, $p_d(x)+p_d(y)\neq p_d(z)+p_d(t)$ implies $x+y\neq z+t$, so $p_d^{-1}(S)$ is a Sidon subset of $A$ of the same cardinality. This shows that $F_N(n) \leq F_1(n)$, and in fact $F_N(n) = F_1(n)$ since the reverse inequality is obvious.

By Lemma \ref{Dirichlet}, we have $F_1(n) = H(n)$ and by Theorem \ref{AmelioAbbott}, we obtain
$$F_N(n)\geqslant\left(\frac{1}{3\sqrt 3}+o(1)\right)\sqrt n \gtrsim 0.19\sqrt n.$$
\end{proof}

\section{Adaptation to $B_2[g]$ sets and generalizations}\label{finale}

\subsection{Result for $B_2[g]$ sets}

We recall that a $B_2[g]$ set is a set such that each element of the ambient semi-group has at most $g$ representations as a sum of two elements of the set, counted up to permutation. In other words, $B$ is a $B_2[g]$ set if, for every $k$, the equation $x+y=k$ has at most $g$ solutions $\left\lbrace x,y\right\rbrace$ with $x,y\in B$ (see \cite{EAB2g} for a recent survey of this notion). For such sets, the proof works in essentially the same way, since one can establish the following analogue of the Singer covering.

\begin{lem}\label{SingerB2}
    Let $q$ be a prime power. There exist $q+1$ $B_2[g]$ sets, each of cardinality $g(q+1)$, covering $\Z/g(q^2+q+1)\Z$.
\end{lem}

\begin{proof}
      Set $N=q^2+q+1$. We start from the Singer covering
      $$\Z/N\Z=\bigcup_{i=1}^{q+1} S_i,$$
      where each $S_i$ is a Sidon set of cardinality $q+1$. Let
      $$\pi:\Z/gN\Z\rightarrow \Z/N\Z$$
      be the canonical projection, and set
      $$S'_i=\pi^{-1}(S_i).$$
      Then
      $$\Z/gN\Z = \bigcup_{i=1}^{q+1} S'_i,$$
      and each $S'_i$ has cardinality $g(q+1)$. We now show that each $S'_i$ is a $B_2[g]$ set. Suppose that
      $$x+x'=k$$
      with $x,x'\in S'_i$. Applying $\pi$, we get
      $$\pi(x)+\pi(x')=\pi(k).$$
      Since $\pi(x),\pi(x')\in S_i$ and $S_i$ is a Sidon set, the unordered pair
      $$\left\lbrace \pi(x),\pi(x')\right\rbrace$$
      is uniquely determined by $\pi(k)$. But $\pi(x)$ and $\pi(x')$ each have exactly $g$ preimages by $\pi$, and the condition $x+x'=k$ makes it so one preimage determines exactly one for the other. Therefore, there are at most $g$ unordered pairs $\{a, b\} \subset S_i' = \pi^{-1}(S_i)$ such that $a+b=k$.

      Therefore $S'_i$ is a $B_2[g]$ set, which concludes the proof.
 \end{proof}
 
 \begin{rmq}
 Proposition \ref{SingerB2} also gives, in particular, the existence of $B_2[g]$ sets of cardinality $\sqrt{gm}$ in the first $m$ integers, but this is already well-known; see for instance \cite{EAB2g}.
 \end{rmq}
 
 \vspace{0.4cm}
 
 The proofs of Theorem \ref{AmelioAbbott}, Lemma \ref{Dirichlet}, and Theorem \ref{MainResult} then adapt without further changes, except that in the proof of Theorem \ref{AmelioAbbott} one now approximates $x=(cn)/g$ by a number of the form $p^2+p+1$. This gives the following theorem.
 
 \begin{theorem}\label{généB2}
 For every fixed $N,g\geqslant 1$, we have
    $$F_{g,N}(n)\geqslant\left(\frac{\sqrt{g}}{3\sqrt 3}+o(1)\right)\sqrt n \gtrsim 0.19\sqrt{gn},$$
    where
    $$F_{g,N}(n)=\min\limits_{\substack{E\subset\R^N \\ |E|=n}}\max\left\lbrace |S| : S\subset E, S \text{ is a } B_2[g]\text{ set}\right\rbrace.$$
 \end{theorem}
 
 \begin{rmq}
 One could try to improve this result, especially in the case $g=2$, for which constructions better than $\sqrt{gm}$ are known; see \cite{PlagneHabsieger}. However, these are constructions in $\left\lbrace 1,\dots,m\right\rbrace$, not in $\Z/m\Z$. They can still be used by embedding $\left\lbrace 1,\dots,m\right\rbrace$ into a cyclic group of larger order and then applying a translation argument, as in Lemma 6 of \cite{Sulyok}, but this leads to a weaker bound.
  \end{rmq}

  
  \subsection{Extension to other linear constraint equations.}
  
  Let us finally mention that the method is not specific to Sidon's equation. It can be adapted to other additive constraints of the form
$$x_1+\cdots+x_k=y_1+\cdots+y_k,$$
or, more generally, to balanced linear equations with positive integer coefficients, after expanding the coefficients as repeated variables. The only modification in the compression step is to replace the set $B_\theta$ in Lemma \ref{Freiman} by a thinner set, for instance
$$B_\theta^{(k)}=\left\lbrace a\in A:\left\lbrace am\theta\right\rbrace<\frac1k\right\rbrace.$$
Then, for any $b_1,\dots,b_k\in B_\theta^{(k)}$, no carry occurs when taking the integer part of
$$(b_1+\cdots+b_k)m\theta.$$
Thus the map
$$\phi_\theta(a)=\left\lfloor am\theta\right\rfloor \mod m$$
is a Freiman $k$-morphism on $B_\theta^{(k)}$. After removing collisions exactly as above, one obtains an injective Freiman $k$-morphism from a large subset of $A$ to $\Z/m\Z$, with a weaker constant coming from the smaller proportion of elements kept in $B_\theta^{(k)}$.

There is, however, one big difference with the Sidon case. For general linear configurations, one usually does not have an analogue of Singer's covering. In that setting, the covering argument can be replaced by the translation trick of Koml\'os, Sulyok and Szemerédi (Lemma 6 in \cite{Sulyok}). Namely, if $B\subset \Z/m\Z$ is a large set avoiding the prescribed configuration, then so is every translate $B+t$, since the equation is balanced. Averaging over $t\in \Z/m\Z$ ensures that some translate of $B$ has large intersection with the compressed set. Pulling this intersection back through the injective Freiman morphism then gives a large subset of the original set avoiding the same configuration. In this way, any sufficiently dense construction in a finite cyclic group can be transferred to arbitrary finite subsets of $\Z$, and then, by the projection and Dirichlet approximation arguments above, to finite subsets of $\R^N$. We do not pursue here the optimization of the constants for these more general equations.
  
\printbibliography

\end{document}